\documentclass{article}
\usepackage{graphicx} 

\usepackage[url=false]{biblatex}
\newcommand{\TZG}{Torba and Zúñiga-Galindo }

\usepackage[letterpaper,top=2cm,bottom=2cm,left=3cm,right=3cm,marginparwidth=1.75cm]{geometry}

\usepackage{amsfonts, amsmath, amssymb, amsthm}
\usepackage{dsfont}
\usepackage{hyperref}

\newcommand{\R}{\mathds{R}}
\newcommand{\Q}{\mathds{Q}}
\newcommand{\Z}{\mathds{Z}}
\newcommand{\Qp}{\Q_p}
\newcommand{\Zp}{\Z_p}

\newcommand{\Adele}{\mathds{A}_\Q}

\newcommand{\floor}[1]{\lfloor #1 \rfloor}
\newcommand{\ceil}[1]{\lceil #1 \rceil}

\newcommand{\Pb}{\mathds{P}}
\newcommand{\Pbp}{\Pb_p}
\newcommand{\Pbpm}{\Pbp^{(m)}}

\newcommand{\Fourier}{\mathcal{F}}
\newcommand{\invFourier}{\Fourier^{-1}}

\newcommand{\pFourier}{\Fourier_p}
\newcommand{\pinvFourier}{\invFourier_p}

\newtheorem{thm}{Theorem}
\newtheorem{prop}{Proposition}
\newtheorem{lemma}{Lemma}
\newtheorem{defn}{Definition}

\newcommand{\Primes}{\mathcal{P}}

\newcommand{\PbAdele}{\Pb_{\sigma, \Adele}}
\newcommand{\PbmAdele}{\Pb_{\sigma, \Adele}^{(m)}}

\newcommand{\PbmAdeleM}{\Pb^{(m)}_{\sigma|_{p < M}, \Adele}}

\newcommand{\Sm}{\mathbf{S}^{(m)}}

\renewcommand{\S}{\mathcal{S}}

\title{A Scaling Limit of Random Walks in the Rational Adeles}
\author{Rahul Rajkumar}
\date{}

\begin{document}

\maketitle

\begin{abstract}
    This paper shows the convergence of adele-valued random walks to an adelic L\'evy process under scaling limits.
    We use random walks on the $p$-adic numbers to construct random walks initially on the infinite product space, and use survival time analysis to prove that the random walks are almost surely adelic for all time.
    The adelic random walks are shown to be small perturbations of processes that are supported on a finite product of path spaces.
    Weak convergence to an adelic L\'evy process is established in the $J_1$ Skorokhod topology.
\end{abstract}

\section{Introduction}

One of the major results of 20th century probability theory is the establishment of the Wiener process as a scaling limit of random walks.
The general study of scaling limits of random walks is a core part of contemporary probability theory, replacing continuous-time continuous-space processes with far more tractable discrete-time discrete-space approximations in a principled manner.
Although the topological disconnectedness of the $p$-adic and adelic state spaces trivializes the space of continuous paths, Skorokhod space and Fourier analytic methods allow for the study of L\'evy processes associated to $p$-adic and adelic (fractional) heat equations.
In this paper, we consider a natural adelic heat equation and its associated L\'evy process.
We obtain an adele-valued random walk process from random walks on each $p$-adic component, and show that these adelic random walks converge weakly to the desired adelic L\'evy process.


The initial motivation for $p$-adic mathematical physics was Volovich's proposal in 1987 that spacetime at the Planck scale may be nonarchimedean \cite{volovich_p-adic_1987}.
Researchers have since studied $p$-adic models as idealizations of a variety of systems displaying hierarchical and ultrametric structure \cite{khrennikov_ultrametric_2018,dragovich_p-adic_2017}, including in the spreading of disease in the presence of social barriers \cite{volov_toward_2020,khrennikov_ultrametric_2020,khrennikov_ultrametric_2020-1}, neural networks and data science \cite{zuniga-galindo_hierarchical_2024,bradley_local_2025},  geophysics \cite{kochubei_linear_2017, khrennikov_application_2016}, and spin glasses \cite{panchenko_parisi_2013,avetisov_p-adic_2002}.
Since it seems unlikely that Nature would distinguish any prime over all the others, adelic models furthermore aim to put all primes on equal footing \cite{varadarajan_remarks_2002, manin_reflections_1989}.
In addition, $p$-adic analogues of classical processes are an active area of research in number theory \cite{van_peski_local_2024,shen_eigenvalues_2026}, while $p$-adic constructions have recently contributed to studying long-range Bernoulli percolation by way of hierarchical lattices \cite{abdesselam_towards_2018,hutchcroft_critical_2025,hutchcroft_critical_2025-1,hutchcroft_critical_2025-2}.

An analogue of the fractional Laplace operator known as the Vladimirov operator, introduced in \cite{vladimirov_generalized_1988}, can be used to define a heat equation over the $p$-adic numbers or a general nonarchimedean local field.
In multiple dimensions and with a specified norm, the Vladimirov operator may be replaced with a Vladimirov-Taibleson operator \cite{taibleson_fourier_1975}.
The fundamental solution determines a L\'evy process \cite{varadarajan_path_1997}, which can also be obtained as scaling limits of random walks \cite{yasuda_limit_2017,weisbart_p-adic_2024,pierce_brownian_2024}.

Weisbart defined an Vladimirov-type operator on the (finite, rational) adeles, which acts as a $p$-adic Vladimirov operator in each component, and showed that the resulting adelic heat equation gives rise to an adelic process associated with the heat equation defined by this operator \cite{weisbart_infinitesimal_2021}.
Urban then extended the results to the setting of an adele ring over an algebraic number field \cite{urban_diffusion_2022}.
Separately, Yasuda investigated adelic Markov processes \cite{yasuda_markov_2010}, using the approach of Albeverio and Karwowski \cite{albeverio_random_1994} to define processes on balls of shrinking radii and related exit times of the process with the Riemann zeta function.
\TZG studied a different adelic process, with dependent components, by constructing an absolute value on the adeles and defining an associated Vladimirov-Taibleson operator \cite{torba_parabolic_2013}.
We use a metric derived from this absolute value in order to utilize standard characterizations of tightness in the Skorokhod topology.

This paper first summarizes some existing work on $p$-adic and adelic processes, including the convergence of certain $p$-adic random walks under a scaling limit.
Some technical details and notation for the $J_1$ topology and weak convergence are presented and specialized to this setting.
A sequence of almost-surely adelic random walks is constructed, whose finite dimensional distributions converge to those of the adelic process.
The general case with a summable sequence of diffusion coefficients is shown to be a small perturbation of random walks with finitely many nonzero diffusion coefficients, and converges weakly to the adelic process.

\section{Preliminaries}

\subsection{$p$-Adic and Adelic Brownian Motion}

A standard introduction to $p$-adic analysis is Gouvêa \cite{gouvea_p-adic_2003}.
For each prime $p$, denote by $\Qp$ the topological field of $p$-adic numbers, equipped with the $p$-adic absolute value $| \cdot |_p$.
Under the metric induced by the $p$-adic absolute value, $\Qp$ is a locally compact, totally disconnected, complete metric space.
The rationals are everywhere dense in $\Qp$, while the integers are a dense subset of the unit ball, which is the topological ring of $p$-adic integers $\Zp$.
Since the underlying additive group $(\Qp, +)$ is a locally compact abelian group, there is a unique Haar measure $\mu_p$ with the normalization $\mu_p(\Zp) = 1$.

Since $\Qp$ is Pontryagin self-dual, identify the Pontryagin dual $\widehat{\Qp}$ with $\Qp$ to define the Fourier and inverse Fourier transforms as integral operators on suitable complex-valued functions $f$ on $\Qp$ against the Haar measure $\mu_p$.
Denote the Fourier and inverse Fourier transforms of $f$ by
\[ \pFourier f  \quad \text{and} \quad \pinvFourier f, \]
again complex-valued functions on $\Qp$.
For any positive real number $b > 0$, let $M_{p,b}$ denote the multiplication operator defined by
\[ (M_{p,b} f) (\cdot) = | \cdot |_p^b f(\cdot). \]
\begin{defn}
    The Vladimirov operator with exponent $b$ is the pseudo-differential operator
    \[ \Delta_{p,b} = \pinvFourier M_{p,b} \pFourier, \]
    a $p$-adic analogue of the fractional Laplace operator.
\end{defn}

\begin{defn}
    Let $\sigma_p$ be a positive real number.
    The $p$-adic heat equation with diffusion coefficient $\sigma_p$ and exponent $b$ is the pseudo-differential equation
    \[ \frac{d}{dt}u = - \sigma_p \Delta_{p,b} u, \]
    where $u(t,x)$ is a function of real time and $p$-adic space.
\end{defn}

\begin{thm}[Varadarajan \cite{varadarajan_path_1997}]
    The fundamental solution to the $p$-adic heat equation with diffusion coefficient $\sigma_p$ and diffusion exponent $b$ is the probability density function $f(t,x; p)$ given by
    \[ f(t,x; p) = \left(\pinvFourier \exp\left(- \sigma_p t | \cdot |_p^b \right) \right)(x). \]
    The family
    \[ \big(f(t, \cdot; p)\big)_{t \in [0,\infty)} \]
    is a convolution semigroup of probability densities on $\Qp$.
\end{thm}

Let $\Qp^{[0,\infty)}$ be the space of all $\Qp$-valued functions on $[0,\infty)$.
Denote by $D(\Qp)$ the subset of $\Qp$-valued paths which are right continuous with left limits, the Skorokhod space of cadlag $\Qp$-valued paths, equipped with the standard $J_1$ Skorokhod topology.
Let $Y^{(p)}$ be the Kolmogorov coordinate function on $[0,\infty) \times \Qp^{[0,\infty)}$, given by
\[ Y^{(p)}(t,\omega) = Y^{(p)}_t(\omega) = \omega(t). \]
Varadarajan uses the Kolmogorov existence theorem to obtain a probability measure on the space $\Qp^{[0,\infty)}$ such that $Y^{(p)}_t$ is a L\'evy process with $Y^{(p)}_0 = 0$ and with increments $Y^{(p)}_{t + u} - Y^{(p)}_u$ having density $f(t,\cdot;p)$ for any $t > 0$ and $u \geq 0$.
By an application of the Chentsov criterion, there is then a version of the stochastic process whose sample paths lie in $D(\Qp)$.
Call such a process a $p$-adic L\'evy process with diffusion coefficient $\sigma_p$ and diffusion exponent $b$.
\begin{thm}[Varadarajan \cite{varadarajan_path_1997}]
    There is a probability measure $\Pbp$ on $D(\Qp)$ such that the stochastic process
    \[ (D(\Qp), \Pbp, Y^{(p)}) \]
    is a $p$-adic L\'evy process with diffusion coefficient $\sigma_p$ and diffusion exponent $b$.
\end{thm}

Let $\Primes$ denote the collection of all primes.
The (finite, rational) adeles are the subring $\Adele$ of the product space $\prod_{p \in \Primes} \Qp$ given by
\[ \Adele = \left\{ x = (x_p)_{p \in \Primes}\in \prod_{p \in \Primes} \Qp : x_p \in \Zp \text{ for almost every $p$} \right\}. \]
Let $\sigma = (\sigma_p)_{p \in \Primes}$ be a summable sequence of nonnegative diffusion coefficients, and let $b > 0$ be fixed.
For suitable functions $f$, define the multiplication operator $M_\sigma$ by
\[ M_\sigma f(x) = \left( \sum_{p \in \Primes} \sigma_p |x_p|_p^b \right) f(x). \]
A Fourier and inverse Fourier transform on $\Adele$ are defined for suitable functions $f = (f_p)_{p \in \Primes}$ by
\[ (\Fourier_{\Adele} f)(x) = \prod_{p \in \Primes} (\pFourier f_p)(x_p) \]
and
\[ (\invFourier_{\Adele} f)(x) = \prod_{p \in \Primes} (\pinvFourier f_p)(x_p). \]
The operator
\[ \Delta_{\sigma, \Adele} = \invFourier_{\Adele} M_\sigma \Fourier_{\Adele} \]
defines an adelic heat equation
\[ \frac{d}{dt} u = -\Delta_{\sigma, \Adele} u. \]
Note that the pseudo-Laplace operator $\Delta_{\sigma, \Adele}$ encodes the diffusion coefficients directly, unlike in the $p$-adic case.

Let $\Pb_{\sigma, \Adele}$ denote the probability measure given by
\[ \Pb_{\sigma, \Adele} = \prod_{p \in \Primes} \Pbp, \]
where $\Pbp$ is the probability measure on $D(\Qp)$ associated with a $p$-adic Brownian motion with diffusion coefficient $\sigma_p$ and diffusion exponent $b$.
The finite dimensional distributions of $\Pb_{\sigma, \Adele}$ are given by integration against the fundamental solution of the adelic heat equation.
Let $D(\Adele)$ denote the subset of $\Adele$-valued paths in the product Skorokhod space \[ D(\Adele) \subset D\left(\prod_{p \in \Primes} \Qp\right) = \prod_{p \in \Primes} D(\Qp). \]

\begin{thm}[Weisbart \cite{weisbart_infinitesimal_2021}]
    The probability measure $\Pb_{\sigma, \Adele}$ gives full measure to $D(\Adele)$ in the product Skorokhod space $D\left(\prod_{p \in \Primes} \Qp\right)$.
\end{thm}
\noindent
Let $\mathbf{Y} = (Y^{(p)})_{p \in \Primes}$ be the Kolmogorov coordinate function.
Refer to the stochastic process
    \[ (D(\Adele), \Pb_{\sigma, \Adele}, \mathbf{Y}) \]
as an adelic L\'evy process with diffusion coefficient (sequence) $\sigma$ and diffusion exponent $b$.

\subsection{Random Walks on $\Qp$}

Let $G_p$ be the quotient group
\[ G_p = \Qp / \Zp, \]
and, for any $x$ in $\Qp$, write $\{x\}$ for the equivalence class $x + \Zp$.
The $p$-adic absolute value passes through the quotient as
\[ |\{x\}|_p = \begin{cases}
    |x|_p & \{x\} \neq \{0\} \\
    0 & \{x\} = \{0\},
\end{cases}\]
under which $G_p$ is a discrete topological group.
For any $p$-adic number $x$, there is a unique coefficient function $a_x : \Z \to \{0, \ldots, p - 1 \}$ whose support is bounded below and satisfies
\[ x = \sum_{k \in \Z} a_x(k) p^k. \]
Identify $G_p$ with its standard embedding in $\Qp$, given by
\[ \{x\} \mapsto \sum_{k < 0} a_x(k) p^k, \]
to obtain a $p$-adic analogue of the standard integer lattice in $\R$.

Define the normalization constant $C_{p,b} = p^b - 1$.
Let $X^{(p)}$ be a $G_p$-valued random variable with the following distribution:
for any positive integer $k$,
\[ \Pr(|X^{(p)}|_p = p^k) = \frac{C_{p,b}}{p^{kb}}, \]
with $X^{(p)}$ uniformly distributed on the level sets $|\{ x \}|_p = p^k$. 
Let $X_i^{(p)} \overset{iid}{\sim} X^{(p)}$ be iid copies of $X^{(p)}$, and let 
\[ S^{(p)}_n = \sum_{i = 1}^n X_i^{(p)} \]
define the (discrete-time) $p$-adic random walk on the embedded $G_p$.
Let $D = \frac{p^b(p - 1)}{p^{b+1} - 1} \sigma_p$ and, for each $m \geq 0$, define the random trajectory $\S^{(m)} : [0, \infty) \to \Qp$ by
\[ \S^{(m)}(t) = p^m S_{\floor{D p^{mb} t}}^{(p)} = p^m \sum_{i = 1}^{\floor{D p^{mb} t}} X_i^{(p)} .\]
The constant $D$ is chosen so that all finite-dimensional distributions of the measures $\Pbpm$ converge to those of a $p$-adic L\'evy process with diffusion coefficient $\sigma_p$ and diffusion exponent $b$.

The following results are from \cite{weisbart_p-adic_2024}, or can be specialized from \cite{pierce_brownian_2024}.
Proposition \ref{prop:moment-estimate} is a moment estimate which implies the uniform Chentsov criterion, and is Proposition 5.1 in \cite{pierce_brownian_2024}.
The uniform Chentsov estimate \cite[Theorem 13.5]{billingsley_convergence_1999} obtained in Proposition \ref{prop:moment-estimate} is a tool used to establish weak convergence in the previous works \cite{weisbart_p-adic_2024, pierce_brownian_2024} which we will be able to reuse to obtain Proposition \ref{prop:finite-support}.

\begin{thm}
    The random trajectories $\S^{(m)}$ induce probability measures $\Pbpm$ on $\Qp^{[0,\infty)}$ which give full measure to $D(\Qp)$.
\end{thm}

\begin{prop}
\label{prop:moment-estimate}
    For any $r$ in $(0,b)$, there is a constant $C$ so that, for any $t$ in $[0,\infty)$,
    \[ \mathds{E}_{\Pbpm}(|Y_t|_p^r) \leq C t^{r/b}, \]
    where $\mathds{E}_{\Pbpm}$ denotes expectation with respect to $\Pbpm$.
\end{prop}

\begin{thm}
    The random walks $\S^{(m)}$ converge weakly to a $p$-adic L\'evy process with diffusion coefficient $\sigma_p$ and diffusion exponent $b$.
\end{thm}

\subsection{Weak Convergence in the Strong $J_1$ Topology}

The product Skorokhod space $\prod_{p \in \Primes} D(\Qp)$ is a Polish space with the product topology, the rarely-used ``weak'' $J_1$ topology \cite{kern_skorokhod_2024}. 
\TZG showed \cite{torba_parabolic_2013} that the adelic absolute value $|\cdot|_{\Adele}$ given by
\[ |x|_{\Adele} = \max_{p \in \Primes} \frac{|x|_p}{p}, \quad x = (x_p)_{p \in \Primes} \in \Adele \]
induces a complete metric on $\Adele$, though neither the weak $J_1$ topology nor the operator $\Delta_{\sigma, \Adele}$ require a choice of metric on $\Adele$.
The adelic metric gives rise to the standard ``strong'' $J_1$ topology, under which $D(\Adele)$ is Polish.


Since both topologies are Polish, and hence sequential, and since a convergent sequence in the strong $J_1$ topology is also convergent in the weak $J_1$ topology, both define the same Borel sigma-algebra \cite[Theorem 11.5.3]{whitt_stochastic-process_2002}.
However, although they support the same collection of probability measures, weak convergence for the weak $J_1$ topology is just weak convergence of the marginals on each factor $D(\Qp)$, while weak convergence in the strong $J_1$ topology is more involved.
Standard theorems for proving weak convergence in the strong $J_1$ topology rely on the metric on the state space \cite[Chapter 16]{billingsley_convergence_1999}.

Instead of relying on a uniform Chentsov estimate for the adelic processes directly, we will use a characterization of uniform tightness for sequences of probability measures.
Let the state space $\Omega$ be either $\Qp$ or $\Adele$, with norm denoted by $|\cdot|_\Omega$.
For an interval $I$ and $\Omega$-valued function $x$, define $w(x,I; \Omega)$ to be
\[ w(x,I ; \Omega) = \sup_{s,t \in I} |x(s) - x(t)|_\Omega. \]
Call a partition $0 = t_0 < \cdots < t_v = T$ of an interval $[0,T)$ essentially $\delta$-sparse if $t_i - t_{i - 1} > \delta$ for $1 \leq i < v$, and let $\Delta(T, \delta)$ denote the collection of all essentially $\delta$-sparse partitions of $[0,T)$.
Define the modified modulus of continuity 
\[ w_T^\prime(x,\delta; \Omega) = \inf_{\Delta(T,\delta)} \max_{1 \leq i \leq v} w(x,[t_{i - 1}, t_i); \Omega). \]
For an $\Adele$-valued function $x = (x_p)_{p \in \Primes}$, exchange suprema to obtain the relation
\[ w_T^\prime(x, \delta; \Adele) = \inf_{\Delta(T,\delta)} \max_{1 \leq i \leq v} \max_{p \in \mathcal{P}} \frac{w(x_p, [t_{i - 1}, t_i); \Qp)}{p}. \]
Tightness of a sequence of probability measures on $D(\Omega)$ with respect to the strong $J_1$ topology is characterized by the modified modulus of continuity and the sup norm for all times $T$.
Convergence of all finite-dimensional distributions then establishes weak convergence.

\begin{thm}[Theorem 16.8 of \cite{billingsley_convergence_1999}]
\label{BillingsleyConditions}
    A sequence of probability measures $(P_n)_{n = 1}^\infty$ is tight in the $J_1$ topology if and only if:
    \begin{enumerate}
        \item For each $T$,
        \[ \lim_{\lambda \to \infty} \limsup_n P_n(x : \sup_{s \leq T} |x(s)|_{\Omega}  \geq \lambda) = 0, \]
        and
        \item for each $T$ and $\lambda$, 
        \[ \lim_{\delta \to 0} \limsup_n P_n(x : w_T^\prime(x, \delta; \Omega) \geq \lambda) = 0.\]
    \end{enumerate}
\end{thm}

\section{Random Walks on the Adeles and Weak Convergence}

Fix a diffusion exponent $b > 0$ throughout.
Let $\sigma = (\sigma_p)_{p \in \Primes}$ be a summable sequence of diffusion coefficients, and for each $m$ let $\Pbpm$ be the random walk measure on $\Qp$ with diffusion coefficient $\sigma_p$.
Initially define the random walks measure 
\[ \PbmAdele = \prod_{p \in \Primes} \Pbpm \]
on the product Skorokhod space $\prod_{p \in \Primes} D(\Qp)$.
Denote the associated random walk process by $\Sm$.
Let $D$ denote the quantity
\[ D \equiv D(p,\sigma, b) = \frac{p^b(p - 1)}{p^{b + 1} - 1} \sigma_p. \]
For any time $t$, write the coordinates of the random walk $\mathbf{S}^{(m)}$ as
\[ \Sm(t) = \left(p^m S^{(p)}_{\floor{D p^{mb} t}}\right)_{p \in \Primes}. \]
For any prime $p$, denote the restriction of the adelic absolute value $| \cdot |_{\Adele}$ to $\Qp$ as
\[ | \cdot |_{p, \Adele} = \frac{|\cdot|_p}{p}.\]
For any positive $\lambda$ and prime $p$, write
\[ [\lambda]_p = p^{\ceil{\log_p(\lambda)} - 1}. \]
Note that, for any integer $k$, $p^k < \lambda$ if and only if $p^k \leq [\lambda]_p$, and
\[ -p^b \lambda^{-b} \leq - [\lambda]_p^{-b} < - \lambda^{-b}. \]
The following lemma computes survival probabilities for the adelic random walks in each component $\Qp$, which will be used to show that each random walk is almost surely valued in the adeles.

\begin{lemma}
\label{lemma:survival-probability-lemma}
    For any $m, T$, and $p$, if $1 \leq m + \ceil{\log_p \lambda}$, then
    \[ \Pbpm \left( \sup_{s \leq T} | p^m S^{(p)}_{\floor{D p^{mb} s}} |_{p, \Adele} < \lambda \right) = \left( 1 - \frac{(p [\lambda]_p)^{-b}}{p^{mb}} \right)^{\floor{D p^{mb} T}}. \]
\end{lemma}

\begin{proof}
    For any positive integer $k$, by the ultrametric property and since $X^{(p)}_i \overset{iid}{\sim} X^{(p)}$ are iid, 
    \begin{align*} 
    \Pr \left( \sup_{j \leq n} |S^{(p)}_j|_p \leq p^k \right) & = \Pr \left( |X^{(p)}|_p \leq p^k \right)^n \\
    & = \left(\sum_{j = 1}^k \frac{p^b - 1}{p^{jb}} \right)^n \\
    & = \left( 1 - p^{-{bk}} \right)^n.
    \end{align*}
The desired probability is then
\begin{align*}
    \Pbpm \left( \sup_{s \leq T} |p^m S^{(p)}_{\floor{D p^{mb} s}}|_{p, \Adele} < \lambda \right) & = \Pbpm \left( \sup_{s \leq T} |p^m S^{(p)}_{\floor{D p^{mb} s}}|_{p, \Adele} \leq [\lambda]_p \right) \\
    & = \Pbpm \left( \sup_{s \leq T} |p^m S^{(p)}_{\floor{D p^{mb} s}}|_p \leq p [\lambda]_p \right) \\
    & = \Pbpm \left( \sup_{s \leq T} |S^{(p)}_{\floor{D p^{mb} s}}|_p \leq p^m p [\lambda]_p \right) \\
    & = \Pbpm \left( \sup_{j \leq \floor{D p^{mb} T}} |S^{(p)}_j|_p \leq p^m p [\lambda]_p \right) \\
    & = \left( 1 - \frac{(p [\lambda]_p)^{-b}}{p^{mb}} \right)^{\floor{D p^{mb} T}}.
\end{align*}
\end{proof}

A path $x = (x_p)_{p \in \Primes}$ in the product Skorokhod space is valued in $\Adele$ up to time $T$ if and only if, for all times $s \leq T$, all but finitely many components $x_p(s)$ are in $\Zp$.
Under the adelic absolute value,
\[ |\Zp|_{p, \Adele} = p^{-1} \]
for every prime $p$, so $x$ is valued in $\Adele$ up to time $T$ if and only if, for every $s \leq T$, 
\[ |x_p(s)|_{p, \Adele} \leq p^{-1} \]
for all but finitely many primes $p$.
Then a sufficient condition for $x$ to be valued in $\Adele$ up to time $T$ is that there exist a finite $M$ such that
\[ \sup_{s \leq T} |x_p(s)|_{p, \Adele} \leq p^{-1} \text{ for all $p \geq M$}. \]

\begin{thm}
    For any $m$, the random walks $\mathbf{S}^{(m)}$ are almost surely valued in $\Adele$ for all time.
\end{thm}

\begin{proof}
    Fix a constant $c > 1$.
    There is a $p_c$ such that, if $p \geq p_c$, then
        \[ \exp(-c p^{-mb}) \leq 1 - \frac{1}{p^{mb}} \leq \exp(- p^{-mb}). \]
    Let $A(T,M; m)$ be the event that for every $p \geq M$ the $p$-th component of $\mathbf{S}^{(m)}$ has not left $\Zp$ up to time $T$:
    \[ A(T,M;m) =  \left\{ \sup_{s \leq T} |p^mS^{(p)}_{\floor{D p^{mb} s}}|_{p, \Adele} \leq p^{-1} \text{ for all $p \geq M$} \right\}. \]
    Let $M \geq p_c$.
    Then with $\lambda = 1$ so that $[\lambda]_p = p^{-1}$, use Lemma \ref{lemma:survival-probability-lemma} to obtain
        \begin{align*}
        \PbmAdele \left( A(T,M; m) \right) & = \prod_{p \geq M} \left( 1 - \frac{1}{p^{mb}} \right)^{\floor{D p^{mb} T}} \\
        & \geq \prod_{p \geq M} \exp \left(-c p^{-mb} D p^{mb} T \right) \\
        & = \prod_{p \geq M} \exp \left(-cT \frac{p^b(p - 1)}{p^{b + 1} - 1} \sigma_p \right) \\
        & = \exp \left( -cT \sum_{p \geq M} \frac{p^b(p - 1)}{p^{b+1} - 1} \sigma_p \right).
    \end{align*}
    Summability of $\sigma_p$ and the limit comparison test imply that $\frac{p^b(p - 1)}{p^{b+1} - 1} \sigma_p$ is summable, and so for each $T$ and $m$,
    \[ \lim_{M \to \infty} \PbmAdele \left( A(T,M; m) \right) = 1. \]
    Continuity from below (in $M$) and above (in $T$) imply $\PbmAdele \left( D(\Adele) \right) = 1$.
\end{proof}

The following proposition handles the special case of a finitely supported diffusion coefficient using the Chentsov criterion as in previous work.
When the diffusion coefficient sequence $\sigma = (\sigma_p)_{p \in \Primes}$ is positive for infinitely many $p$, the adelic absolute value will depend on infinitely many components.
However, the summability of $\sigma$ will ensure that the random walks are small perturbations of the special case, circumventing this issue.

\begin{prop}
\label{prop:finite-support}
    Assume that $\sigma$ is finitely supported, so that there is some $M$ such that $\sigma_p > 0$ implies $p < M$.
    Then the random walks $\mathbf{S}^{(m)}$ converge weakly to an adelic L\'evy process with diffusion coefficient $\sigma$.
\end{prop}
\begin{proof}   

    For each $p < M$ and any $0 < r < b$, by the uniform Chentsov estimate of Proposition \ref{prop:moment-estimate} there is a constant $C_p$ such that
    \[ \mathds{E}_{\Pbpm} |Y^{(p)}_t|_p^r < C_p t^{r/b}. \]
    Then, for some $C_M$,
    \[ \mathds{E}_{\PbmAdele} |\mathbf{Y}_t|^r_{\Adele} \leq M \max_{p < M} \mathds{E}_{\Pbpm} \frac{|Y_t^{(p)}|_p^r}{p^r} < C_M t^{r/b},  \]
    and so the random walk measures $\PbmAdele$ satisfy the uniform Chentsov criterion.
    Convergence of the finite-dimensional distributions to those of $\PbAdele$ along with the uniform Chentsov condition then implies the weak convergence of the random walks to adelic Brownian motion.
    
\end{proof}

For a path 
\[ x(t) = (x_p(t))_{p \in \Primes},\]
write $x|_{p < M}$ to denote the path
\[ x|_{p < M} (t) = (x_p(t) \mathds{1}_{p < M} )_{p \in \Primes}. \]
Similarly, write $\sigma|_{p < M}$ to denote the finitely supported diffusion coefficient
\[ \sigma|_{p < M} = (\sigma_p \mathds{1}_{p < M})_{p \in \Primes}. \]
Let $\Lambda(T,M,\lambda)$ be the event
\[ \Lambda(T, M, \lambda) = \left\{ \forall p \geq M, \sup_{s \leq T} |x_p(s)|_{p, \Adele} < \lambda \right\}. \]
The following lemmas will be used to prove the second condition of Theorem \ref{BillingsleyConditions}.
\begin{lemma}
\label{lemma:dominating-probability}
For any $T, m, \lambda$, and $M$,
    \[ \PbmAdele \left( w_T^\prime(x, \delta; \Adele) \geq \lambda \cap \Lambda(T, M, \lambda) \right) \leq \PbmAdeleM(w_T^\prime(x|_{p < M}, \delta; \Adele) \geq \lambda). \]   
\end{lemma}

\begin{proof}
If $w_T^\prime(x, \delta; \Adele) \geq \lambda$, then for every essentially $\delta$-sparse $\{ t_i \}$ there exists an interval $[t_{i - 1}, t_i)$ such that
\[ \max_{p \in \Primes} \frac{w(x_p, [t_{i - 1}, t_i); \Qp)}{p} \geq \lambda. \]
If moreover $\sup_{s \leq T}|x_p(s)|_{p, \Adele} < \lambda$ for all $p \geq M$, then 
\[ \frac{w(x_p, [t_{i - 1}, t_i); \Qp)}{p} < \lambda \]
for all subintervals $[t_{i - 1}, t_i)$ and $p \geq M$.
Hence
    \[ \lambda \leq \max_{p \in \Primes} \frac{w(x_p, [t_{i - 1}, t_i); \Qp)}{p} = \max_{p < M} \frac{w(x_p, [t_{i - 1}, t_i); \Qp)}{p} = w({x|_{p < M}}, [t_{i - 1}, t_i); \Adele), \] 
and so, together,  
\[ w_T^\prime(x, \delta; \Adele) \geq \lambda \quad \text{ and } \quad \sup_{s \leq T} |x_p(s)|_{p, \Adele} < \lambda \text{ for all $p \geq M$}\]
implies 
\[ w_T^\prime(x|_{p < M}, \delta; \Adele) \geq \lambda. \]
The result then follows from the equality
\[ \PbmAdele( w_T^\prime(x|_{p < M}, \delta; \Adele) \geq \lambda)  =  \PbmAdeleM( w_T^\prime(x|_{p < M}, \delta; \Adele) \geq \lambda). \]
\end{proof}

\begin{lemma}
\label{lemma:small-tail-probability}
    For any $T, \lambda$, and $\epsilon$, there is an $M = M(T,\lambda, \epsilon)$ such that
    \[ \liminf_m \PbmAdele \left( \Lambda(T, M, \lambda) \right) > 1 - \epsilon . \] 
\end{lemma}

\begin{proof}
    For any $T, \lambda, \epsilon$, and $M$,
    \[ \liminf_m \PbmAdele \left( \Lambda(T, M, \lambda) \right) = \lim_m \prod_{p \geq M} \left( 1 - \frac{(p [\lambda]_p)^{-b}}{p^{mb}} \right)^{\floor{D p^{mb} T}}. \]
    Apply the dominated convergence theorem for products to interchange the limits and obtain
    \begin{align*}
    \liminf_m \PbmAdele \left( \Lambda(T, M, \lambda) \right)  &= \prod_{p \geq M} \lim_m \left( 1 - \frac{(p [\lambda]_p)^{-b}}{p^{mb}} \right)^{\floor{D p^{mb}T}} \\
        & = \prod_{p \geq M} \exp \left( -T [\lambda]_p^{-b} \frac{p - 1}{p^{b + 1} - 1} \sigma_p \right). 
    \end{align*}
    Since
    \[ -p^b \lambda^{-b} \leq - [\lambda]_p^{-b} < - \lambda^{-b}, \]
    obtain
    \begin{align*}
    \liminf_m \PbmAdele \left( \Lambda(T, M, \lambda) \right) & \geq \prod_{p \geq M} \exp \left( -T \lambda^{-b} \frac{p^b(p - 1)}{p^{b+1} - 1} \sigma_p \right) \\
    & = \exp \left( -T \lambda^{-b} \sum_{p \geq M} \frac{p^b(p - 1)}{p^{b+1} - 1} \sigma_p \right).
    \end{align*}
    Since the sum is finite, 
    \[ \lim_{M \to \infty} \exp \left( -T \lambda^{-b} \sum_{p \geq M} \frac{p^b(p - 1)}{p^{b+1} - 1} \sigma_p \right) = 1, \]
    so there is then a sufficiently large $M$ such that 
    \[ \liminf_m \PbmAdele \left( \Lambda(T, M, \lambda) \right) > 1 - \epsilon . \] 
\end{proof}

The following propositions prove that the tightness conditions of Theorem \ref{BillingsleyConditions} hold.
\begin{prop}
\label{prop:prove-billingsley-2}
    For each $T$ and $\lambda$, 
    \[ \lim_{\delta \to 0} \limsup_m \PbmAdele (w_T^\prime(x, \delta; \Adele) \geq \lambda) = 0. \]
\end{prop}

\begin{proof}

Use monotonicity and the law of total probability with the events $\Lambda(T, M, \lambda)$ and its complement to obtain
\begin{align*}
    \PbmAdele (w_T^\prime(x, \delta; \Adele) \geq \lambda) 
        & \leq \PbmAdele( w_T^\prime(x, \delta; \Adele) \geq \lambda \cap \Lambda(T,M,\lambda)) \\
         & \qquad + \left( 1 - \PbmAdele(\Lambda(T,M,\lambda)) \right).
\end{align*}
By Lemma \ref{lemma:dominating-probability},
\[ \PbmAdele( w_T^\prime(x, \delta; \Adele) \geq \lambda \cap \Lambda(T,M,\lambda)) \leq \PbmAdeleM(w_T^\prime(x|_{p < M}, \delta; \Adele) \geq \lambda).   \]
By Proposition \ref{prop:finite-support}, the measures $\PbmAdeleM$ converge weakly, and hence are tight.
Then by the conditions of Theorem \ref{BillingsleyConditions}, 
\[ \lim_{\delta \to 0} \limsup_m \PbmAdeleM(w_T^\prime(x|_{p < M}, \delta; \Adele) \geq \lambda) = 0.   \]
By Lemma \ref{lemma:small-tail-probability}, for any $\epsilon$ there is an $M$ large enough so that 
\[  \limsup_m \left( 1 - \PbmAdele(\Lambda(T,M,\lambda)) \right) < \epsilon.  \]
Hence
\[ \lim_{\delta \to 0} \limsup_m \PbmAdele(w_T^\prime(x, \delta; \Adele) \geq \lambda) < \epsilon \]
for any $\epsilon > 0$.
\end{proof}

\begin{prop}
\label{prop:prove-billingsley-1}
For any $T$,
    \[ \lim_{\lambda \to \infty} \liminf_m \PbmAdele ( \sup_{s \leq T} |x(s)|_{\Adele} < \lambda) = 1. \]
\end{prop}

\begin{proof}
    Since
    \[ -p^b \lambda^{-b} \leq - [\lambda]_p^{-b} < - \lambda^{-b}, \]
    apply the dominated convergence theorem for products to obtain
    \begin{align*}
         \liminf_m \prod_{p \in \Primes} \left( 1 - \frac{(p [\lambda]_p)^{-b}}{p^{mb}} \right)^{\floor{D p^{mb} T}} & = \prod_p \exp \left( - (p[\lambda]_p)^{-b} \frac{p^b(p - 1)}{p^{b+1} - 1}\sigma_pT \right) \\
         & \geq \prod_{p \in \Primes} \exp \left( - \lambda^{-b} \frac{p^b(p - 1)}{p^{b + 1} - 1} \sigma_p T \right)  \\
         & = \exp \left( - \lambda^{-b} T \sum_{p \in \Primes} \frac{p^b(p -1)}{p^{b + 1} - 1} \sigma_p \right).
    \end{align*}
    Since the sum is finite, take the limit as $\lambda$ goes to infinity to conclude 
    \[ \lim_{\lambda \to \infty} \liminf_m \PbmAdele ( \sup_{s \leq T} |x(s)|_{\Adele} < \lambda) = 1. \]
\end{proof}

\begin{thm}
The random walks $\mathbf{S}^{(m)}$ converge weakly to an adelic L\'evy process with diffusion coefficient $\sigma$ and diffusion exponent $b$.
\end{thm}
\begin{proof}
    Propositions \ref{prop:prove-billingsley-2} and \ref{prop:prove-billingsley-1} prove that the measures $\PbmAdele$ corresponding to the adelic random walks $\mathbf{S}^{(m)}$ satisfy the conditions of Theorem \ref{BillingsleyConditions}, and therefore are tight.
    Convergence of the finite-dimensional distributions then implies weak convergence.
\end{proof}

\printbibliography

@article{avetisov_p-adic_2002,
	title = {p-{Adic} {Models} of {Ultrametric} {Diffusion} {Constrained} by {Hierarchical} {Energy} {Landscapes}},
	volume = {35},
	issn = {0305-4470, 1361-6447},
	url = {http://arxiv.org/abs/cond-mat/0106506},
	doi = {10.1088/0305-4470/35/2/301},
	language = {en},
	number = {2},
	urldate = {2023-11-06},
	journal = {Journal of Physics A: Mathematical and General},
	author = {Avetisov, V. A. and Bikulov, A. H. and Kozyrev, S. V. and Osipov, V. A.},
	month = jan,
	year = {2002},
	note = {arXiv:cond-mat/0106506},
	pages = {177--189},
}

@book{khrennikov_ultrametric_2018,
	edition = {1},
	title = {Ultrametric {Pseudodifferential} {Equations} and {Applications}:},
	isbn = {978-1-107-18882-2 978-1-316-98670-7},
	shorttitle = {Ultrametric {Pseudodifferential} {Equations} and {Applications}},
	url = {https://www.cambridge.org/core/product/identifier/9781316986707/type/book},
	doi = {10.1017/9781316986707},
	urldate = {2023-08-07},
	publisher = {Cambridge University Press},
	author = {Khrennikov, Andrei Yu. and Kozyrev, Sergei V. and Zúñiga-Galindo, W. A.},
	month = apr,
	year = {2018},
}

@article{dragovich_p-adic_2017,
	title = {\$p\$-{Adic} {Mathematical} {Physics}: {The} {First} 30 {Years}},
	volume = {9},
	issn = {2070-0466, 2070-0474},
	shorttitle = {\$p\$-{Adic} {Mathematical} {Physics}},
	url = {http://arxiv.org/abs/1705.04758},
	doi = {10.1134/S2070046617020017},
	language = {en},
	number = {2},
	urldate = {2023-03-04},
	journal = {p-Adic Numbers, Ultrametric Analysis and Applications},
	author = {Dragovich, B. and Khrennikov, A. Yu and Kozyrev, S. V. and Volovich, I. V. and Zelenov, E. I.},
	month = apr,
	year = {2017},
	note = {arXiv:1705.04758 [hep-th, physics:math-ph, q-bio]},
	pages = {87--121},
}

@misc{hutchcroft_critical_2025,
	title = {Critical long-range percolation {III}: {The} upper critical dimension},
	shorttitle = {Critical long-range percolation {III}},
	url = {http://arxiv.org/abs/2508.18809},
	doi = {10.48550/arXiv.2508.18809},
	urldate = {2025-09-13},
	publisher = {arXiv},
	author = {Hutchcroft, Tom},
	month = aug,
	year = {2025},
	note = {arXiv:2508.18809 [math]},
}

@misc{hutchcroft_critical_2025-1,
	title = {Critical long-range percolation {I}: {High} effective dimension},
	shorttitle = {Critical long-range percolation {I}},
	url = {http://arxiv.org/abs/2508.18807},
	doi = {10.48550/arXiv.2508.18807},
	urldate = {2025-09-13},
	publisher = {arXiv},
	author = {Hutchcroft, Tom},
	month = aug,
	year = {2025},
	note = {arXiv:2508.18807 [math]},
}

@article{bradley_local_2025,
	title = {On the {Local} {Ultrametricity} of {Finite} {Metric} {Data}},
	issn = {1432-1343},
	url = {https://doi.org/10.1007/s00357-025-09508-3},
	doi = {10.1007/s00357-025-09508-3},
	language = {en},
	urldate = {2025-09-12},
	journal = {Journal of Classification},
	author = {Bradley, Patrick Erik},
	month = mar,
	year = {2025},
}

@article{zuniga-galindo_hierarchical_2024,
	title = {Hierarchical {Neural} {Networks}, p-{Adic} {PDEs}, and {Applications} to {Image} {Processing}},
	volume = {31},
	issn = {1776-0852},
	url = {https://doi.org/10.1007/s44198-024-00229-6},
	doi = {10.1007/s44198-024-00229-6},
	language = {en},
	number = {1},
	urldate = {2025-09-12},
	journal = {Journal of Nonlinear Mathematical Physics},
	author = {Zúñiga-Galindo, W. A. and Zambrano-Luna, B. A. and Dibba, Baboucarr},
	month = sep,
	year = {2024},
	pages = {63},
}

@article{panchenko_parisi_2013,
	title = {The {Parisi} ultrametricity conjecture},
	volume = {177},
	url = {https://annals.math.princeton.edu/2013/177-1/p08},
	language = {en-US},
	urldate = {2025-09-12},
	journal = {Annals of Mathematics},
	author = {Panchenko, Dmitry},
	year = {2013},
	pages = {383--393},
}

@misc{kochubei_linear_2017,
	title = {Linear and {Nonlinear} {Heat} {Equations} on a p-{Adic} {Ball}},
	url = {http://arxiv.org/abs/1708.03261},
	doi = {10.48550/arXiv.1708.03261},
	language = {en},
	urldate = {2025-09-12},
	publisher = {arXiv},
	author = {Kochubei, Anatoly N.},
	month = aug,
	year = {2017},
	note = {arXiv:1708.03261 [math]},
}

@misc{khrennikov_ultrametric_2020,
	title = {Ultrametric model for covid-19 dynamics: an attempt to explain slow approaching herd immunity in {Sweden}},
	copyright = {© 2020, Posted by Cold Spring Harbor Laboratory. This pre-print is available under a Creative Commons License (Attribution-NoDerivs 4.0 International), CC BY-ND 4.0, as described at http://creativecommons.org/licenses/by-nd/4.0/},
	shorttitle = {Ultrametric model for covid-19 dynamics},
	url = {https://www.medrxiv.org/content/10.1101/2020.07.04.20146209v1},
	doi = {10.1101/2020.07.04.20146209},
	language = {en},
	urldate = {2025-09-12},
	publisher = {medRxiv},
	author = {Khrennikov, Andrei},
	month = jul,
	year = {2020},
	note = {Pages: 2020.07.04.20146209},
}

@misc{abdesselam_towards_2018,
	title = {Towards three-dimensional conformal probability},
	url = {http://arxiv.org/abs/1511.03180},
	doi = {10.1134/S207004661804001},
	urldate = {2025-09-10},
	author = {Abdesselam, Abdelmalek},
	month = sep,
	year = {2018},
	note = {arXiv:1511.03180 [math]},
}

@misc{hutchcroft_critical_2025-2,
	title = {Critical long-range percolation {II}: {Low} effective dimension},
	shorttitle = {Critical long-range percolation {II}},
	url = {http://arxiv.org/abs/2508.18808},
	doi = {10.48550/arXiv.2508.18808},
	language = {en},
	urldate = {2025-09-10},
	publisher = {arXiv},
	author = {Hutchcroft, Tom},
	month = aug,
	year = {2025},
	note = {arXiv:2508.18808 [math]},
}

@book{taibleson_fourier_1975,
	address = {Princeton, N.J},
	series = {Mathematical notes ; no. 15},
	title = {Fourier analysis on local fields},
	isbn = {978-0-691-08165-6},
	publisher = {Princeton University Press},
	author = {Taibleson, M. H.},
	year = {1975},
}

@article{khrennikov_application_2016,
	title = {Application of p-{Adic} {Wavelets} to {Model} {Reaction}–{Diffusion} {Dynamics} in {Random} {Porous} {Media}},
	volume = {22},
	issn = {1069-5869, 1531-5851},
	url = {http://link.springer.com/10.1007/s00041-015-9433-y},
	doi = {10.1007/s00041-015-9433-y},
	language = {en},
	number = {4},
	urldate = {2025-06-10},
	journal = {Journal of Fourier Analysis and Applications},
	author = {Khrennikov, Andrei and Oleschko, Klaudia and De Jesús Correa López, Maria},
	month = aug,
	year = {2016},
	pages = {809--822},
}

@book{gouvea_p-adic_2003,
	address = {Berlin ; New York},
	edition = {2nd ed},
	series = {Universitext},
	title = {p-adic numbers: an introduction},
	isbn = {978-3-540-62911-5},
	shorttitle = {p-adic numbers},
	publisher = {Springer},
	author = {Gouvêa, Fernando Q.},
	year = {2003},
}

@article{volov_toward_2020,
	title = {Toward {Ultrametric} {Modeling} of the {Epidemic} {Spread}},
	volume = {12},
	issn = {2070-0474},
	url = {https://doi.org/10.1134/S2070046620030061},
	doi = {10.1134/S2070046620030061},
	language = {en},
	number = {3},
	urldate = {2025-02-12},
	journal = {p-Adic Numbers, Ultrametric Analysis and Applications},
	author = {Volov, V. T. and Zubarev, A. P.},
	month = jul,
	year = {2020},
	pages = {247--258},
}

@article{albeverio_random_1994,
	title = {A random walk on p-adics—the generator and its spectrum},
	volume = {53},
	copyright = {https://www.elsevier.com/tdm/userlicense/1.0/},
	issn = {03044149},
	url = {https://linkinghub.elsevier.com/retrieve/pii/030441499490054X},
	doi = {10.1016/0304-4149(94)90054-X},
	language = {en},
	number = {1},
	urldate = {2024-10-28},
	journal = {Stochastic Processes and their Applications},
	author = {Albeverio, Sergio and Karwowski, Witold},
	month = sep,
	year = {1994},
	pages = {1--22},
}

@article{varadarajan_path_1997,
	title = {Path {Integrals} for a {Class} of {P}-{Adic} {Schrödinger} {Equations}},
	volume = {39},
	issn = {1573-0530},
	url = {https://doi.org/10.1023/A:1007364631796},
	doi = {10.1023/A:1007364631796},
	language = {en},
	number = {2},
	journal = {Letters in Mathematical Physics},
	author = {Varadarajan, V. S.},
	month = jan,
	year = {1997},
	pages = {97--106},
}

@article{pierce_brownian_2024,
	title = {Brownian motion in a vector space over a local field is a scaling limit},
	volume = {42},
	issn = {0723-0869},
	url = {https://www.sciencedirect.com/science/article/pii/S0723086924000744},
	doi = {10.1016/j.exmath.2024.125607},
	number = {6},
	journal = {Expositiones Mathematicae},
	author = {Pierce, Tyler and Rajkumar, Rahul and Stine, Andrea and Weisbart, David and Yassine, Adam M.},
	month = dec,
	year = {2024},
	pages = {125607},
}

@article{weisbart_p-adic_2024,
	title = {\$p\$-{Adic} {Brownian} {Motion} is a {Scaling} {Limit}},
	volume = {57},
	issn = {1751-8113, 1751-8121},
	url = {http://arxiv.org/abs/2010.05492},
	doi = {10.1088/1751-8121/ad40df},
	number = {20},
	journal = {Journal of Physics A: Mathematical and Theoretical},
	author = {Weisbart, David},
	month = may,
	year = {2024},
	note = {arXiv:2010.05492 [math-ph]},
	pages = {205203},
}

@article{urban_diffusion_2022,
	title = {On a diffusion on finite adeles and the {Feynman}-{Kac} integral},
	volume = {63},
	issn = {0022-2488},
	url = {https://doi.org/10.1063/5.0111423},
	doi = {10.1063/5.0111423},
	number = {12},
	urldate = {2025-11-24},
	journal = {Journal of Mathematical Physics},
	author = {Urban, Roman},
	month = dec,
	year = {2022},
	pages = {122101},
}

@article{weisbart_infinitesimal_2021,
	title = {On infinitesimal generators and {Feynman}–{Kac} integrals of adelic diffusion},
	volume = {62},
	issn = {0022-2488, 1089-7658},
	url = {https://pubs.aip.org/jmp/article/62/10/103504/318342/On-infinitesimal-generators-and-Feynman-Kac},
	doi = {10.1063/5.0056119},
	language = {en},
	number = {10},
	urldate = {2025-11-24},
	journal = {Journal of Mathematical Physics},
	author = {Weisbart, David},
	month = oct,
	year = {2021},
	pages = {103504},
}

@book{billingsley_convergence_1999,
	address = {New York Weinheim},
	edition = {2. ed},
	series = {Wiley series in probability and statistics {Probability} and statistics section},
	title = {Convergence of probability measures},
	isbn = {978-0-471-19745-4 978-0-470-31780-8},
	language = {en},
	publisher = {Wiley},
	author = {Billingsley, Patrick},
	year = {1999},
}

@article{torba_parabolic_2013,
	title = {Parabolic {Type} {Equations} and {Markov} {Stochastic} {Processes} on {Adeles}},
	volume = {19},
	copyright = {http://www.springer.com/tdm},
	issn = {1069-5869, 1531-5851},
	url = {http://link.springer.com/10.1007/s00041-013-9277-2},
	doi = {10.1007/s00041-013-9277-2},
	language = {en},
	number = {4},
	urldate = {2025-12-05},
	journal = {Journal of Fourier Analysis and Applications},
	author = {Torba, S. M. and Zúñiga-Galindo, W. A.},
	month = aug,
	year = {2013},
	pages = {792--835},
}

@article{varadarajan_remarks_2002,
	title = {Some remarks on arithmetic physics},
	volume = {103},
	copyright = {https://www.elsevier.com/tdm/userlicense/1.0/},
	issn = {03783758},
	url = {https://linkinghub.elsevier.com/retrieve/pii/S037837580100194X},
	doi = {10.1016/S0378-3758(01)00194-X},
	language = {en},
	number = {1-2},
	urldate = {2025-12-17},
	journal = {Journal of Statistical Planning and Inference},
	author = {Varadarajan, V.S.},
	month = apr,
	year = {2002},
	pages = {3--13},
}

@article{kern_skorokhod_2024,
	title = {Skorokhod topologies},
	volume = {71},
	issn = {1432-1815},
	url = {https://doi.org/10.1007/s00591-023-00353-2},
	doi = {10.1007/s00591-023-00353-2},
	language = {en},
	number = {1},
	urldate = {2025-12-17},
	journal = {Mathematische Semesterberichte},
	author = {Kern, Julian},
	month = mar,
	year = {2024},
	pages = {1--18},
}

@article{van_peski_local_2024,
	title = {Local limits in p\$p\$‐adic random matrix theory},
	volume = {129},
	issn = {0024-6115, 1460-244X},
	url = {https://londmathsoc.onlinelibrary.wiley.com/doi/10.1112/plms.12626},
	doi = {10.1112/plms.12626},
	language = {en},
	number = {3},
	urldate = {2026-02-25},
	journal = {Proceedings of the London Mathematical Society},
	author = {Van Peski, Roger},
	month = sep,
	year = {2024},
	pages = {e12626},
}

@misc{shen_eigenvalues_2026,
	title = {Eigenvalues of \$p\$-adic random matrices},
	url = {http://arxiv.org/abs/2601.06283},
	doi = {10.48550/arXiv.2601.06283},
	urldate = {2026-02-25},
	publisher = {arXiv},
	author = {Shen, Jiahe and Peski, Roger Van},
	month = jan,
	year = {2026},
	note = {arXiv:2601.06283 [math]},
}

@incollection{manin_reflections_1989,
	title = {{REFLECTIONS} {ON} {ARITHMETICAL} {PHYSICS}},
	copyright = {https://www.elsevier.com/tdm/userlicense/1.0/},
	isbn = {978-0-12-218100-9},
	url = {https://linkinghub.elsevier.com/retrieve/pii/B9780122181009500170},
	doi = {10.1016/B978-0-12-218100-9.50017-0},
	language = {en},
	urldate = {2026-02-26},
	booktitle = {Conformal {Invariance} and {String} {Theory}},
	publisher = {Elsevier},
	author = {Manin, Yu.I.},
	year = {1989},
	pages = {293--303},
}

@article{volovich_p-adic_1987,
	title = {p-adic space-time and string theory},
	volume = {71},
	copyright = {http://www.springer.com/tdm},
	issn = {0040-5779, 1573-9333},
	url = {http://link.springer.com/10.1007/BF01017088},
	doi = {10.1007/BF01017088},
	language = {en},
	number = {3},
	urldate = {2026-02-26},
	journal = {Theoretical and Mathematical Physics},
	author = {Volovich, I. V.},
	month = jun,
	year = {1987},
	pages = {574--576},
}

@book{whitt_stochastic-process_2002,
	address = {New York, NY},
	series = {Springer {Series} in {Operations} {Research} and {Financial} {Engineering}},
	title = {Stochastic-{Process} {Limits}: {An} {Introduction} to {Stochastic}-{Process} {Limits} and {Their} {Application} to {Queues}},
	copyright = {http://www.springer.com/tdm},
	isbn = {978-0-387-95358-8 978-0-387-21748-2},
	shorttitle = {Stochastic-{Process} {Limits}},
	url = {http://link.springer.com/10.1007/b97479},
	doi = {10.1007/b97479},
	language = {en},
	urldate = {2026-02-26},
	publisher = {Springer New York},
	author = {Whitt, Ward},
	editor = {Glynn, Peter W. and Robinson, Stephen M.},
	year = {2002},
}

@article{khrennikov_ultrametric_2020-1,
	title = {An {Ultrametric} {Random} {Walk} {Model} for {Disease} {Spread} {Taking} into {Account} {Social} {Clustering} of the {Population}},
	volume = {22},
	issn = {1099-4300},
	url = {https://www.mdpi.com/1099-4300/22/9/931},
	doi = {10.3390/e22090931},
	language = {en},
	number = {9},
	urldate = {2026-02-26},
	journal = {Entropy},
	author = {Khrennikov, Andrei and Oleschko, Klaudia},
	month = aug,
	year = {2020},
	pages = {931},
}

@article{vladimirov_generalized_1988,
	title = {Generalized functions over the field of \textit{p} -adic numbers},
	volume = {43},
	issn = {0036-0279, 1468-4829},
	url = {https://www.mathnet.ru/eng/rm1967},
	doi = {10.1070/RM1988v043n05ABEH001924},
	number = {5},
	urldate = {2026-02-26},
	journal = {Russian Mathematical Surveys},
	author = {Vladimirov, Vasilii S},
	month = oct,
	year = {1988},
	pages = {19--64},
}

@article{yasuda_markov_2010,
	title = {Markov {Processes} on the {Adeles} and {Representations} of {Euler} {Products}},
	volume = {23},
	issn = {1572-9230},
	url = {https://doi.org/10.1007/s10959-009-0222-x},
	doi = {10.1007/s10959-009-0222-x},
	language = {en},
	number = {3},
	urldate = {2026-04-23},
	journal = {Journal of Theoretical Probability},
	author = {Yasuda, Kumi},
	month = sep,
	year = {2010},
	pages = {748--769},
}

@article{yasuda_limit_2017,
	title = {Limit theorems for p-adic valued asymmetric semistable laws and processes},
	volume = {9},
	issn = {2070-0474},
	url = {https://doi.org/10.1134/S207004661701006X},
	doi = {10.1134/S207004661701006X},
	language = {en},
	number = {1},
	urldate = {2026-04-24},
	journal = {p-Adic Numbers, Ultrametric Analysis and Applications},
	author = {Yasuda, Kumi},
	month = jan,
	year = {2017},
	pages = {62--77},
}

\end{document}